\documentclass[12pt,reqno]{amsart}
\usepackage{amsmath,amssymb,extarrows}
\usepackage{url}
\usepackage{tikz,enumerate}
\usepackage{diagbox}
\usepackage{appendix}
\usepackage{epic}

\usepackage{float} %support many floats

\usepackage{cite}

\usepackage{hyperref}
\usepackage{array}

\usepackage{booktabs}

\allowdisplaybreaks[4]

\newtheorem{theorem}{Theorem}

\newtheorem{lemma}[theorem]{Lemma}

\theoremstyle{definition}

\theoremstyle{remark}
\newtheorem{rem}{Remark}
\numberwithin{equation}{section}
\numberwithin{theorem}{section}
\numberwithin{defn}{section}
\newcommand*\diff{\mathop{}\!\mathrm{d}}

\begin{document}
\title[New Proofs of Some Double Sum Rogers-Ramanujan Type Identities]
 {New Proofs of Some Double Sum Rogers-Ramanujan Type Identities}

\author{Liuquan Wang}
\address{School of Mathematics and Statistics, Wuhan University, Wuhan 430072, Hubei, People's Republic of China}
\email{wanglq@whu.edu.cn;mathlqwang@163.com}



\begin{abstract}
Recently, Rosengren utilized an integral method to prove a number of conjectural identities found by Kanade and Russell. Using this integral method, we give new proofs to some double sum identities of Rogers-Ramanujan type. These identities were earlier proved by approaches such as combinatorial arguments or by using $q$-difference equations. Our proofs are based on streamlined calculations, which relate these double sum identities to some known Rogers-Ramanujan type identities with single sums. Moreover, we prove a conjectural identity of Andrews and Uncu which was earlier confirmed by Chern.
\end{abstract}

\maketitle
%\tableofcontents

\section{Introduction and main results}\label{sec-intro}
The famous Rogers-Ramanujan identities state that
\begin{align}\label{RR}
\sum_{n=0}^\infty \frac{q^{n^2}}{(q;q)_n}=\frac{1}{(q,q^4;q^5)_\infty}, \quad \sum_{n=0}^\infty \frac{q^{n(n+1)}}{(q;q)_n}=\frac{1}{(q^2,q^3;q^5)_\infty}.
\end{align}
Here and throughout this paper, we assume that $|q|<1$ for convergence and use the standard $q$-series notation
\begin{align}
&(a;q)_0:=1, \quad (a;q)_n:=\prod\limits_{k=0}^{n-1}(1-aq^k),  \\
%\quad (a;q)_\infty :=\prod\limits_{k=0}^\infty (1-aq^k),
&(a_1,\cdots,a_m;q)_n:=(a_1;q)_n\cdots (a_m;q)_n, \quad n\in \mathbb{N}\cup \{\infty\}.
\end{align}

Despite their analytic forms given in \eqref{RR}, the Rogers-Ramanujan identities can also be formulated as combinatorial partition identities.  The $i$-th ($i=1,2$) identity is equivalent to saying that the number of partitions of $n$  such that the adjacent parts differ by least 2 and such that the smallest part is at least $i$ is the same as the number of partitions of $n$ into parts congruent to $\pm i $ modulo 5.

Since the appearances of \eqref{RR}, lots of similar identities have been found and they were called as Rogers-Ramanujan type identities. One of the famous works on this topic is Slater's list \cite{Slater}, which contains 130 such identities. For example, the following identities are in Slater's list \cite[Eqs.\ (34),(36)]{Slater}
\begin{align}
\sum_{n=0}^\infty \frac{q^{n^2+2n}(-q;q^2)_n}{(q^2;q^2)_n}&=\frac{1}{(q^3,q^4,q^5;q^8)_\infty}, \label{Slater34} \\
\sum_{n=0}^\infty \frac{q^{n^2}(-q;q^2)_n}{(q^2;q^2)_n}&=\frac{1}{(q,q^4,q^7;q^8)_\infty}. \label{Slater36}
\end{align}
Both identities were earlier found by Ramanujan \cite[Entries 1.17.11, 1.7.12, 4.2.15]{Lost2}.
G\"{o}llnitz \cite[Eqs.\ (2.22), (2.24)]{Gollnitz} found the following companion identities:
\begin{align}
\sum_{n=0}^\infty \frac{q^{n(n+1)}(-q^{-1};q^2)_n}{(q^2;q^2)_n}&=\frac{1}{(q,q^5,q^6;q^8)_\infty}, \label{Gollnitz22} \\
\sum_{n=0}^\infty \frac{q^{n(n+1)}(-q;q^2)_n}{(q^2;q^2)_n}&=\frac{1}{(q^2,q^3,q^7;q^8)_\infty}. \label{Gollnitz24}
\end{align}

There is now a huge list of Rogers-Ramanujan type identities, and it is evidently that more and more will be discovered in the future. We refer the reader to Sills' book \cite{Sills-book} on this topic for more information.

One of the way to generalize the Rogers-Ramanujan identities is to consider multi-sum identities. For instance, the Andrews-Gordon identities state that
\begin{align}
\sum_{n_1,\dots,n_{k-1}\geq 0}\frac{q^{N_1^2+\cdots+N_{k-1}^2+N_i+\cdots+N_{k-1}}}{(q;q)_{n_1}\cdots (q;q)_{n_{k-1}}}=\frac{(q^{i},q^{2k+1-i},q^{2k+1};q^{2k+1})_\infty}{(q;q)_\infty}, \label{AG}
\end{align}
where $1\leq i\leq k$, $k\geq 2$ and $N_j=n_j+\cdots+n_{k-1}$. This identity was given by Andrews \cite{Andrews1974}, and is the analytic counterpart of Gordon's \cite{Gordon1961} combinatorial generalization of \eqref{RR}.

In recent years, multi-sum identities of the Rogers-Ramanujan type have drawn lots of attention. In particular, a number of recent works studied Rogers-Ramanujan type identities in the form
\begin{align}\label{index-defn}
\text{finite sum of } \sum_{(i_1,\cdots,i_k)\in S}\frac{(-1)^{t(i_1,\cdots,i_k)}q^{Q(i_1,\cdots,i_k)}}{(q^{n_1};q^{n_1})_{i_1}\cdots (q^{n_k};q^{n_k})_{i_k}}=\prod\limits_{ (a,n)\in P} (q^{a};q^n)_\infty^{r(a,n)}.
\end{align}
Here $t(i_1,\cdots,i_k)$ is an integer-valued function, $Q(i_1,\cdots,i_k)$ is a rational polynomial in variables $i_1,\cdots,i_k$, $n_i$ ($i=1,2,\cdots,k$) are positive integers with $\gcd(n_1,\cdots,n_k)$ $=1$, $P$ is a finite subset of $\mathbb{Q}^2$ and $r(a,n)$ are integer-valued functions.

We shall call any identity of the form \eqref{index-defn} a Rogers-Ramanujan type identity with {\it index} $(n_1,\cdots,n_k)$. Note that here we do not allow $q$-Pochhammer symbols to appear in the numerator in each summand. But as we can see from \eqref{Slater34}--\eqref{Gollnitz24}, this may be allowed if we discuss in a more general context.

Around 2019, Kanade and Russell \cite{KR-2019} searched for Rogers-Ramanujan type identities related to level 2 characters of the affine Lie algebra $A_9^{(2)}$, and they conjectured a number of such identities. Let
\begin{align}
F(u,v,w):=\sum_{i,j,k=0}^\infty \frac{(-1)^kq^{3k(k-1)+(i+2j+3k)(i+2j+3k-1)}u^iv^jw^k}{(q;q)_i(q^4;q^4)_j(q^6;q^6)_k},
\end{align}
which is roughly speaking a triple sum of index $(1,4,6)$ according to our definition. Some of their conjectural identities are
\begin{align}
%F(q,1,q^3)&=\frac{(q^3;q^{12})_\infty}{(q,q^2;q^4)_\infty}, \label{KR-triple-0} \\
F(q,q^6,q^9)&=\frac{1}{(q,q^4,q^6,q^8,q^{11};q^{12})_\infty}, \label{KR-triple-1} \\
F(q,q,q^6)&=\frac{1}{(q^3;q^4)_\infty (q,q^8;q^{12})_\infty}. \label{KR-tripke-2}
\end{align}
Five of their conjectural identities were proved by Bringmann, Jennings-Shaffer and Mahlburg \cite{BSM}. In a recent work, Rosengren \cite{Rosengren} presented a streamlined approach and proved all of the nine conjectural identities on $F(u,v,w)$.

There are two steps in Rosengren's approach. First, one uses an integral representation of the triple sum to reduce it to a single sum. Second, one needs to evaluate the single sum via $q$-series techniques. The difficulty in the second step varies a lot for different identities. We shall call this approach as the integral method. With this method, Rosengren also gave a new proof of the following identity with index $(1,3)$:
\begin{align}
\sum_{i,j\geq 0}\frac{q^{2i^2+6ij+6j^2}}{(q;q)_i(q^3;q^3)_j}=\frac{1}{(q^3;q^6)_\infty (q^2,q^{10};q^{12})_\infty}. \label{AKKK}
\end{align}
As shown in \cite{KR-2019} and \cite{Kursungoz}, this identity is equivalent to a conjectural partition identity of Capparelli \cite{Capparelli}, which was first proved by Andrews \cite{Andrews1994}.

There are more double sum Rogers-Ramanujan type identities in the literature and their proofs differ from each other. Some of them can be converted to a single sum in a straightforward way. For example,  Takigiku and Tsuchioka \cite{Takigiku2019} proved three identities with index $(1,2)$ such as
\begin{align}\label{eq-Takigiku}
\sum_{i,j\geq 0}\frac{(-1)^jq^{\binom{i}{2}+2\binom{j}{2}+2ij+i+j}}{(q;q)_i(q^2;q^2)_j}=\frac{1}{(q^2,q^3,q^4,q^{10},q^{11},q^{12};q^{14})_\infty}.
\end{align}
Their original proofs of these identities use $q$-difference equations. But they also mentioned \cite[Remark 4.5]{Takigiku2019} that, by summing over $j$ first using Euler's identity
\begin{align}
\sum_{n=0}^\infty \frac{q^{\binom{n}{2}}z^n}{(q;q)_n}=(-z;q)_\infty, \label{Euler-1}
\end{align}
the sum side in \eqref{eq-Takigiku} can easily be shown equal to the left side of Slater's identity \cite[Eq.\ (81)]{Slater}
\begin{align}
\sum_{i=0}^\infty \frac{q^{\frac{1}{2}(i^2+i)}}{(q;q)_i(q;q^2)_i}=\frac{(q,q^6,q^7;q^7)_\infty (q^5,q^9;q^{14})_\infty}{(q;q)_\infty (q;q^2)_\infty},
\end{align}
and hence \eqref{eq-Takigiku} holds.

The same trick was applied by Andrews in \cite{Andrews2019} and he found three more identities with index $(1,2)$ such as
\begin{align}
\sum_{i,j\geq 0}\frac{(-1)^jq^{(i+j)^2}}{(q;q)_i(q^2;q^2)_j}=\frac{1}{(q^2,q^4,q^{10},q^{12};q^{14})_\infty}. \label{Andrews-double}
\end{align}

We will focus on double sum Rogers-Ramanujan type identities in which the sum indexes cannot be separated or eliminated directly.
Clearly \eqref{AKKK} is such an identity. We can find some other such identities in the literature. In a recent work, Uncu and Zudilin \cite{Uncu-Zudilin} discovered two identities with index $(1,2)$ (see \eqref{UZ-1} and \eqref{UZ-2} below). In the way of proving their identities, we find a companion identity \eqref{UZ-3} and we state them together as follows.
\begin{theorem}\label{thm-UZ}
We have
\begin{align}
\sum_{i,j\geq 0}\frac{q^{i^2+2ij+2j^2}}{(q;q)_i(q^2;q^2)_j}&=\frac{(q^3;q^3)_\infty^2}{(q;q)_\infty(q^6;q^6)_\infty}, \label{UZ-1} \\
\sum_{i,j\geq 0}\frac{q^{i^2+2ij+2j^2+i+2j}}{(q;q)_i(q^2;q^2)_j}&=\frac{(q^6;q^6)_\infty^2}{(q^2;q^2)_\infty (q^3;q^3)_\infty}, \label{UZ-2} \\
\sum_{i,j\geq 0}\frac{q^{i^2+2ij+2j^2+j}}{(q;q)_i(q^2;q^2)_j}&=(-q;q)_\infty. \label{UZ-3}
\end{align}
\end{theorem}
Uncu and Zudilin \cite{Uncu-Zudilin}  pointed out that, as explained by Ole Warnaar, these two identities are instances of Bressoud's identities from \cite{Bressoud1979}.  We will provide new proofs, which relate \eqref{UZ-1} to Ramanujan's identity \cite[p.\ 87, Entry 4.2.11]{Lost2}:
\begin{align}
\sum_{n=0}^\infty \frac{(-1)^n(q;q^2)_nq^{n(n+2)}}{(q^4;q^4)_n}=\frac{(q^3;q^3)_\infty (q^{12};q^{12})_\infty}{(q^4;q^4)_\infty (q^6;q^6)_\infty}, \label{Rama-1}
\end{align}
and relate \eqref{UZ-2} to the following identities of Ramanujan \cite[pp.\ 88, 101, Entries 4.2.13 and 5.3.2]{Lost2}:
\begin{align}
\sum_{n=0}^\infty \frac{(-q^2;q^2)_nq^{n(n+1)}}{(q;q)_{2n+1}}&=\frac{(q^3;q^3)_\infty(q^{12};q^{12})_\infty}{(q;q)_\infty (q^6;q^6)_\infty }, \label{Rama-2} \\
\sum_{n=0}^\infty \frac{(-q;q^2)_n q^{n^2}}{(q;q)_{2n}}&=\frac{(q^6;q^6)_\infty^2}{(q;q)_\infty  (q^{12};q^{12})_\infty}. \label{Rama-3}
\end{align}
The identity \eqref{UZ-3} appears to be new. We will see that it is a consequence of \eqref{Euler-1}.

In 2021, through $q$-difference equations, Andrews and Uncu \cite[Theorem 1.1]{Andrews-Uncu} proved the following interesting identity with index $(1,3)$.
\begin{theorem}\label{thm-AU}
\begin{align}\label{Andrews-Uncu-id}
\sum_{i,j\geq 0}\frac{(-1)^jq^{i^2+3ij+3j(3j+1)/2}}{(q;q)_i(q^3;q^3)_j}=\frac{1}{(q;q^3)_\infty}.
\end{align}
\end{theorem}
They also conjectured that \cite[Conjecture 1.2]{Andrews-Uncu}
\begin{align}\label{AU-conj}
\sum_{i,j\geq 0}\frac{(-1)^jq^{\frac{3j(3j+1)}{2}+i^2+3ij+i+j}}{(q;q)_i(q^3;q^3)_j}=\frac{1}{(q^2,q^3;q^6)_\infty}.
\end{align}
Let $\omega=e^{2\pi i/3}$ throughout this paper. We find that \eqref{Andrews-Uncu-id} can be proved through the following Rogers-Ramanujan type identity
\begin{align}\label{AU-RR}
\sum_{n=0}^\infty \frac{(-1)^n\omega^nq^{\frac{1}{2}n(n+1)}}{(q,\omega^2 q;q)_n}=\frac{1}{(q;q^3)_\infty},
\end{align}
which we will prove in this paper. Furthermore, we find that in order to prove the conjectural identity \eqref{AU-conj}, it suffices to show the following  identity.
\begin{theorem}\label{thm-Wangconj}
We have
\begin{align}\label{eq-Wang}
\sum_{n=0}^\infty \frac{(-1)^nq^{\frac{3}{2}n^2+\frac{5}{2}n}(q;q^3)_n}{(q^9;q^9)_n}=\frac{(q^4;q^6)_\infty (q^{12};q^{18})_\infty}{(q^5;q^6)_\infty (q^9;q^{18})_\infty}.
\end{align}
\end{theorem}
\begin{rem}
This theorem was posed as a conjecture in the first version of this paper (appeared as arXiv: 2203.15572v1), which contains essentially all materials of this paper except that Theorem \ref{thm-Wangconj} is unproven. After the first draft has been written, we heard from Shane Chern \cite{Chern-communication} that he has independently worked on the identity \eqref{AU-conj} and also discovered \eqref{eq-Wang}. He \cite{Chern} has successfully proved \eqref{eq-Wang} by proving the identity
\begin{align}\label{Chern-gen}
\sum_{n=0}^\infty \frac{(a;q)_n(a^{-1}q^2;q^2)_n}{(a^2q;q^2)_n(q^3;q^3)_n}(-1)^n a^nq^{n(n+1)/2}=\frac{(aq;q^2)_\infty (a^3q^3;q^6)_\infty}{(a^2q;q^2)_\infty (q^3;q^6)_\infty}.
\end{align}
The above theorem follows by replacing $q$ with $q^3$ and setting $a=q$ in \eqref{Chern-gen}.
\end{rem}
Note that Chern proved \eqref{Chern-gen} by proving a finite sum version of it,  we will give a different proof of Theorem \ref{thm-Wangconj} by pointing out that \eqref{Chern-gen} is in fact a direct consequence of a cubic summation formula due to Rahman \cite{Rahman}. As a result, this yields another proof of the conjectural identity \eqref{AU-conj}.
\begin{theorem}\label{thm-conj}
The identity \eqref{AU-conj} holds.
\end{theorem}

Finally, we discuss a set of identities with index $(1,4)$, which can be found in the work of Kur\c{s}ung\"{o}z \cite[Corollary 18]{Kursungoz}.
\begin{theorem}\label{thm14}
We have
\begin{align}
\sum_{i,j\geq 0}\frac{q^{(3i^2-i)/2+4ij+4j^2}}{(q;q)_{i}(q^4;q^4)_{j}}&=\frac{1}{(q,q^4,q^7;q^8)_\infty}, \label{KR21} \\
\sum_{i,j\geq 0}\frac{q^{(3i^2+3i)/2+4ij+4j^2+4j}}{(q;q)_{i}(q^4;q^4)_{j}}&=\frac{1}{(q^3,q^4,q^5;q^8)_\infty}, \label{KR22} \\
\sum_{i,j\geq 0}\frac{q^{(3i^2+i)/2+4ij+4j^2+2j}(1+q^{2i+4j+1})}{(q;q)_{i}(q^4;q^4)_{j}}&=\frac{1}{(q,q^5,q^6;q^8)_\infty}, \label{KR23} \\
\sum_{i,j\geq 0}\frac{q^{(3i^2+i)/2+4ij+4j^2+2j}}{(q;q)_{i}(q^4;q^4)_{j}}&=\frac{1}{(q^2,q^3,q^7;q^8)_\infty}. \label{KR24}
\end{align}
\end{theorem}
We will provide new proofs of these identities using the integral method, and by this we see that they are related to the identities \eqref{Slater34}--\eqref{Gollnitz24}.

The paper is organized as follows. In Section \ref{sec-pre} we collect some useful $q$-series formulas which will be used to derive the identities. In Section \ref{sec-proof} we present new proofs for Theorems \ref{thm-UZ}, \ref{thm-AU}, \ref{thm14} and give a proof for Theorem \ref{thm-conj}. Finally, in Section \ref{sec-conclude} we propose some conjectural identities with index $(1,2,3,4)$.

\section{Preliminaries}\label{sec-pre}
Besides \eqref{Euler-1}, we also need another identity due to Euler
\begin{align}\label{Euler-2}
\sum_{n=0}^\infty \frac{z^n}{(q;q)_n}=\frac{1}{(z;q)_n}, \quad |z|<1
\end{align}
and the Jacobi triple product identity
\begin{align}\label{Jacobi}
(q,z,q/z;q)_\infty=\sum_{n=-\infty}^\infty (-1)^nq^{\binom{n}{2}}z^n.
\end{align}
They will be used frequently to express multi-sums as integrals of infinite products. Both \eqref{Euler-1} and \eqref{Euler-2} can be deduced from the $q$-binomial theorem: for $|q|,|z|<1$,
\begin{align}\label{q-binomial}
\sum_{n=0}^\infty \frac{(a;q)_n}{(q;q)_n}z^n=\frac{(az;q)_\infty}{(z;q)_\infty}.
\end{align}

One of the key step in establishing Rogers-Ramanujan type identities by the integral method is to evaluate integrals of infinite $q$-products  and express the result as infinite $q$-products. We will mainly use an identity from the book of Gasper and Rahman \cite{GR-book}. Before we state this result, we remark that the symbol ``idem $(c_1;c_2,\dots,c_C)$'' after an expression stands for the sum of the $(C-1)$ expressions obtained from the preceding expression by interchanging $c_1$ with each $c_k$, $k=2,3,\dots,C$.

\begin{lemma}\label{lem-integral}
(Cf.\ \cite[Eq.\ (4.10.6)]{GR-book})
Suppose that
$$P(z):=\frac{(a_1z,\dots,a_Az,b_1/z,\dots,b_B/z;q)_\infty}{(c_1z,\dots,c_Cz,d_1/z,\dots,d_D/z;q)_\infty}$$
has only simple poles. We have
\begin{align}\label{eq-integral}
\oint P(z)\frac{dz}{2\pi iz}=& \frac{(b_1c_1,\dots,b_Bc_1,a_1/c_1,\dots,a_A/c_1;q)_\infty }{(q,d_1c_1,\dots,d_Dc_1,c_2/c_1,\dots,c_C/c_1;q)_\infty} \nonumber \\
& \times \sum_{n=0}^\infty \frac{(d_1c_1,\dots,d_Dc_1,qc_1/a_1,\dots,qc_1/a_A;q)_n}{(q,b_1c_1,\dots,b_Bc_1,qc_1/c_2,\dots,qc_1/c_C;q)_n} \nonumber \\
&\times \Big(-c_1q^{(n+1)/2}\Big)^{n(C-A)}\Big(\frac{a_1\cdots a_A}{c_1\cdots c_C} \Big)^n +\text{idem} ~(c_1;c_2,\dots,c_C)
\end{align}
when $C>A$, or if $C=A$ and
\begin{align}\label{cond}
\left|\frac{a_1\cdots a_A }{c_1\cdots c_C}\right|<1.
\end{align}
Here the integration is over a positively oriented contour so that the poles of $(c_1z,\dots,c_Cz;q)_\infty^{-1}$ lie outside the contour, and the origin and poles of $(d_1/z,\dots,d_D/z;q)_\infty^{-1}$ lie inside the contour.
\end{lemma}

Heine's transformation formula \cite[p.\ 13, Eq.\ (1.4.3)]{GR-book} states that for $|t|<1$,
\begin{align}\label{Heine}
\sum_{n=0}^\infty \frac{(a,b;q)_n}{(c,q;q)_n}t^n=\frac{(abt/c;q)_\infty}{(t;q)_\infty} \sum_{n=0}^\infty \frac{(\frac{c}{a},\frac{c}{b};q)_n}{(c,q;q)_n}\left(\frac{abt}{c}\right)^n.
\end{align}
We also need Lebesuge's identity (see \cite[p.\ 21, Cor.\ 2.7]{Andrews}, for example)
\begin{align}\label{Lebesgue}
\sum_{n=0}^\infty \frac{q^{n(n+1)/2}(a;q)_n}{(q;q)_n}=(aq;q^2)_\infty (-q;q)_\infty.
\end{align}
We will occasionally use the notation $f_m:=(q^m;q^m)_\infty$ when doing calculations. The following lemma will be used in some proofs.
\begin{lemma}\label{lem-dissection}
We have
\begin{align}
\frac{f_3}{f_1^3}&=\frac{f_4^6f_6^3}{f_2^9f_{12}^2}+3q\frac{f_4^2f_6f_{12}^2}{f_2^7}, \label{f13-1} \\
\frac{f_1^3}{f_3}&=\frac{f_4^3}{f_{12}}-3q\frac{f_2^2f_{12}^3}{f_4f_6^2}. \label{f13-2}
\end{align}
\end{lemma}
Proofs of \eqref{f13-1} and \eqref{f13-2} can be found in \cite{Xia} and \cite{Hirschhorn}, respectively.

\section{Proofs of the theorems}\label{sec-proof}

In this section, we will present proofs of the main results stated in Section \ref{sec-intro}. Throughout this section, the contours of integrals will be chosen in the way as in Lemma \ref{lem-integral}.

\begin{proof}[Proof of Theorem \ref{thm-UZ}]
We shall consider
\begin{align}
F_1(u,v;q):=\oint \frac{(uz,qz,1/z,q;q)_\infty}{(vz,-vz;q)_\infty} \frac{dz}{2\pi iz}.
\end{align}
By \eqref{Euler-1} and \eqref{Euler-2} we see that
\begin{align}
F_1(u,v;q)&=\oint \sum_{i=0}^\infty \frac{(-1)^i q^{\binom{i}{2}}u^iz^i}{(q;q)_i} \sum_{j=0}^\infty \frac{v^{2j}z^{2j}}{(q^2;q^2)_j} \sum_{k=-\infty}^\infty (-1)^kq^{\binom{k}{2}}z^{-k} \frac{dz}{2\pi iz} \nonumber \\
&=\sum_{i,j\geq 0}\frac{u^iv^{2j}q^{i^2+2ij+2j^2-i-j}}{(q;q)_i (q^2;q^2)_j}. \label{F1-exp}
\end{align}

To prove \eqref{UZ-1}, we choose $u=q$ and $v=q^{\frac{1}{2}}$. Setting
$$(A,B,C,D)=(2,1,2,0), \quad a_1=a_2=q, \quad b_1=1, \quad (c_1,c_2)=(q^{1/2}, -q^{1/2})$$
in \eqref{eq-integral}, we deduce that
\begin{align}
F_1(q,q^{1/2};q)=(q;q)_\infty \left(\frac{(q^{\frac{1}{2}};q)_\infty^3}{(q,-1;q)_\infty} S_1(q)+\frac{(-q^{\frac{1}{2}};q)_\infty^3}{(q,-1;q)_\infty}S_2(q)\right),
\end{align}
where
$$S_1(q)=\sum_{n=0}^\infty \frac{(q^{\frac{1}{2}};q)_n(-q)^n}{(q,-q;q)_n}, \quad S_2(q)=\sum_{n=0}^\infty \frac{(-q^{\frac{1}{2}};q)_n(-q)^n}{(q,-q;q)_n}.$$
By Heine's transformation \eqref{Heine} with $(q,a,b,c,t)\rightarrow (q^2,q,0,-q^2,-q^2)$, we obtain
\begin{align}
S_1(q^2)=\frac{1}{(-q^2;q^2)_\infty}\sum_{n=0}^\infty \frac{(-q;q^2)_n q^{n^2+2n}}{(-q^2,q^2;q^2)_n}=\frac{f_2f_6^2}{f_3f_4^2},
\end{align}
where the last equality follows from \eqref{Rama-1} with $q$ replaced by $-q$.

Replacing $q$ by $-q$ in the expression of $S_1(q^2)$ we get $S_2(q^2)$ and hence
\begin{align}
S_2(q^2)=\frac{f_2f_3f_{12}}{f_4^2f_6}.
\end{align}
Hence
%\begin{align}
%F_1(q^2,-q;q^2)&=(q^2;q^2)_\infty \left(\frac{(q;q^2)_\infty^3}{2(q^4;q^4)_\infty}S_1(q^2)+\frac{(-q;q^2)_\infty^3}{2(q^4;q^4)_\infty}S_2(q^2) \right) \nonumber  \\
%&=\frac{1}{2}(q^2;q^2)_\infty \left(\frac{(q;q)_\infty^3}{(q^3;q^3)_\infty} \cdot \frac{(q^6;q^6)_\infty^2}{(q^2;q^2)_\infty^2(q^4;q^4)_\infty^3} +\frac{(q^3;q^3)_\infty}{(q;q)_\infty^3}\cdot \frac{(q^2;q^2)_\infty^7 (q^{12};q^{12})_\infty}{(q^4;q^4)_\infty ^6(q^6;q^6)_\infty}  \right) \nonumber \\
%&=\frac{(q^6;q^6)_\infty^2}{(q^2;q^2)_\infty (q^{12};q^{12})_\infty}.
%\end{align}

\begin{align}
F_1(q^2,q;q^2)&=(q^2;q^2)_\infty \left(\frac{(q;q^2)_\infty^3}{2(q^4;q^4)_\infty}S_1(q^2)+\frac{(-q;q^2)_\infty^3}{2(q^4;q^4)_\infty}S_2(q^2) \right) \nonumber  \\
&=\frac{1}{2}f_2 \left(\frac{f_1^3}{f_3} \cdot \frac{f_6^2}{f_2^2f_4^3} +\frac{f_3}{f_1^3}\cdot \frac{f_2^7f_{12}}{f_4^6f_6}  \right) \nonumber \\
&=\frac{f_6^2}{f_2f_{12}}.
\end{align}
Here the last equality follows by substitution of \eqref{f13-1} and \eqref{f13-2} and simplification. Replacing $q$ by $q^{\frac{1}{2}}$, we obtain \eqref{UZ-1}.

To prove \eqref{UZ-2}, we need to evaluate $F_1(q^2,q^{\frac{3}{2}};q)$. But in this case Lemma \ref{lem-integral} is not applicable since $C=A$ and the condition \eqref{cond} does not hold. Instead,  we shall use the integral
\begin{align}
F_2(u,v;q):=\oint \frac{(vz,q^2z,1/z,q^2;q^2)_\infty}{(uz;q)_\infty} \frac{dz}{2\pi iz}.
\end{align}
By \eqref{Euler-1}, \eqref{Euler-2} and \eqref{Jacobi} we have
\begin{align}
F_2(u,v;q)&=\oint \sum_{i=0}^\infty \frac{u^iz^i}{(q;q)_i} \sum_{j=0}^\infty \frac{(-1)^jz^jv^jq^{j^2-j}}{(q^2;q^2)_j}\sum_{\ell=-\infty}^\infty (-1)^\ell q^{\ell^2-\ell}z^{-\ell} \frac{dz}{2\pi iz} \nonumber \\
&=\sum_{i,j\geq 0} \frac{(-1)^iu^iv^jq^{i^2+2ij+2j^2-i-2j}}{(q;q)_i(q^2;q^2)_j}.
\end{align}
Now we set $u=-q^2,v=q^4$. Replacing $q$ by $q^2$ in \eqref{eq-integral} and letting
$$(A,B,C,D)=(2,1,2,0), \quad (a_1,a_2)=(q^2, q^4), \quad b_1=1, \quad (c_1,c_2)=(-q^2, -q^3),$$
we deduce that
\begin{align}\label{F2-start}
F_2(-q^2,q^4;q)=2\frac{(-q^2;q^2)_\infty^3}{(q;q)_\infty} T_1(q)-\frac{(-q;q^2)_\infty^3}{(q;q)_\infty}T_2(q),
\end{align}
where
\begin{align}
T_1(q)=\sum_{n=0}^\infty \frac{(-1;q^2)_nq^n}{(q;q)_{2n}}, \quad T_2(q)=\sum_{n=0}^\infty \frac{(-q;q^2)_nq^n}{(q;q)_{2n+1}}.
\end{align}
Now by Heine's transformation \eqref{Heine} with $(q,a,b,c,t)\rightarrow (q^2,-1,0,q,q) $ and \eqref{Rama-3}, we deduce that
\begin{align}
T_1(q)=\frac{1}{(q;q^2)_\infty}\sum_{n=0}^\infty \frac{(-q;q^2)_nq^{n^2}}{(q,q^2;q^2)_n}=\frac{f_2f_6^2}{f_1^2f_{12}}. \label{T1-result}
\end{align}
Similarly, using \eqref{Heine} with $(q,a,b,c,t)\rightarrow (q^2,-q,0,q^3,q)$ and \eqref{Rama-2}, we deduce that
\begin{align}
T_2(q)=\frac{1}{(q;q^2)_\infty} \sum_{n=0}^\infty \frac{(-q^2;q^2)_nq^{n^2+n}}{(q;q)_{2n+1}}=\frac{f_2f_3f_{12}}{f_1^2f_6}. \label{T2-result}
\end{align}
Now substituting \eqref{T1-result} and \eqref{T2-result} into \eqref{F2-start},  we deuce that
\begin{align*}
F_2(-q^2,q^4;q)&=2\frac{f_4^3f_6^2}{f_1^3f_2^2f_{12}}-\frac{f_2^7f_3f_{12}}{f_1^6f_4^3f_6} \\
&=\frac{1}{f_1^3}\left(2\frac{f_4^3f_6^2}{f_2^2f_{12}}-\frac{f_2^7f_{12}}{f_4^3f_6}\cdot \frac{f_3}{f_1^3} \right)  \quad \text{(use \eqref{f13-1})}\\
&=\frac{1}{f_1^3}\left(2\frac{f_4^3f_6^2}{f_2^2f_{12}}-\frac{f_2^7f_{12}}{f_4^3f_6}\cdot \Big( \frac{f_4^6f_6^3}{f_2^9f_{12}^2}+3q\frac{f_4^2f_6f_{12}^2}{f_2^7} \Big) \right) \nonumber \\
&=\frac{f_6^2}{f_1^3f_2^2}\left(\frac{f_4^3}{f_{12}}-3q\frac{f_2^2f_{12}^3}{f_4f_6^2}  \right) \quad \text{(by \eqref{f13-2})} \nonumber \\
&=\frac{f_6^2}{f_2^2f_3}.
\end{align*}

Finally, to prove \eqref{UZ-3}, setting $u=v=q$ in \eqref{F1-exp}, we see that the left side of \eqref{UZ-3} is
\begin{align*}
F_1(q,q;q)&=\oint \frac{(qz,1/z,q;q)_\infty}{(-qz;q)_\infty}\frac{dz}{2\pi iz} \nonumber \\
&=\oint \sum_{k=0}^\infty \frac{(-qz)^k}{(q;q)_k} \sum_{\ell=-\infty}^\infty (-1)^\ell q^{\binom{\ell}{2}}z^{-\ell} \frac{dz}{2\pi iz} \\
&=\sum_{k=0}^\infty \frac{q^{k(k+1)/2}}{(q;q)_k}=(-q;q)_\infty.
\end{align*}
Here the second and last equalities follow from \eqref{Euler-1} and \eqref{Euler-2}.
\end{proof}
\begin{rem}
For \eqref{UZ-3}, we cannot use \eqref{eq-integral} to evaluate $F_1(q,q;q)$ directly since $A=C=1$ and $|a_1/c_1|=1$.
\end{rem}

\begin{proof}[Proof of Theorem \ref{thm-AU}]
We start with the integral
\begin{align}
R(q)&:=\oint \sum_{i=0}^\infty \frac{q^{\frac{1}{2}(i^2+i)}(-1)^iz^i}{(q;q)_i} \sum_{j=0}^\infty \frac{q^{3j}z^{3j}}{(q^3;q^3)_j} \sum_{k=-\infty}^\infty (-1)^k q^{\frac{1}{2}(k^2-k)}z^{-k} \frac{dz}{2\pi iz} \nonumber \\
&=\sum_{i,j\geq 0} \frac{(-1)^jq^{i^2+3ij+\frac{3}{2}j(3j+1)}}{(q;q)_i(q^3;q^3)_j}.
\end{align}
By \eqref{Euler-1}, \eqref{Euler-2} and \eqref{Jacobi}, we can write
\begin{align}
R(q)=\oint \frac{(qz,qz,1/z,q;q)_\infty}{(q^3z^3;q^3)_\infty}\frac{dz}{2\pi iz}  =\oint \frac{(qz,1/z,q;q)_\infty}{(\omega qz,\omega^2qz;q)_\infty}\frac{dz}{2\pi iz}.
\end{align}
Now applying \eqref{eq-integral} with
$$(A,B,C,D)=(1,1,2,0), \quad a_1=q, \quad b_1=1, \quad (c_1,c_2)=(\omega q, \omega^2q),$$
we deduce that
\begin{align}
R(q)=\frac{(\omega^2;q)_\infty}{1-\omega}R_1(q)+\frac{(\omega;q)_\infty}{1-\omega^2} R_2(q), \label{R-exp}
\end{align}
where
\begin{align}
R_1(q)&=\sum_{n=0}^\infty \frac{(-1)^n \omega^n q^{\frac{1}{2}n(n+1)}}{(q,\omega^2q;q)_n}, \\
R_2(q)&=\sum_{n=0}^\infty \frac{(-1)^n\omega^{2n}q^{\frac{1}{2}n(n+1)}}{(q,\omega q;q)_n}.
\end{align}
From \cite[Eq.\ (6.1.11)]{Sills} we find
\begin{align}
\sum_{n=0}^\infty \frac{q^{\frac{1}{2}n(n-1)}(-b)^n}{(a,q;q)_n}=(b;q)_\infty \sum_{n=0}^\infty \frac{q^{\frac{1}{2}(3n^2-3n)}(-ab)^n}{(q,a,b;q)_n}.
\end{align}
Setting $a=\omega^2q$ and $b=\omega q$ in the above identity and using \eqref{Euler-1}, we obtain
\begin{align}
R_1(q)=(\omega q;q)_\infty \sum_{n=0}^\infty \frac{(-1)^n q^{\frac{1}{2}(3n^2+n)}}{(q^3;q^3)_n}=(\omega q;q)_\infty (q^2;q^3)_\infty. \label{R1-exp}
\end{align}
Replacing $\omega$ by $\omega^2$, we obtain
\begin{align}
R_2(q)=(\omega^2q;q)_\infty (q^2;q^3)_\infty. \label{R2-exp}
\end{align}
Substituting \eqref{R1-exp} and \eqref{R2-exp} into \eqref{R-exp}, we deduce that
\begin{align*}
R(q)=(q^2;q^3)_\infty (\omega q;q)_\infty (\omega^2q;q)_\infty \left( \frac{1-\omega^2}{1-\omega}+\frac{1-\omega}{1-\omega^2} \right)=\frac{1}{(q;q^3)_\infty}.
\end{align*}
\end{proof}

\begin{proof}[Proof of Theorem \ref{thm-Wangconj}]
Recall the following formula found by Rahman \cite{Rahman}:
\begin{align}
&\sum_{n=0}^\infty \frac{(1-aq^{5n})(a,b;q^2)_n(ab^2/q^3;q^3)_n(q^2/b;q)_n(aq^3/b;q^6)_n}{(1-a)(q^3,aq^3/b;q^3)_n(q^5/b^2;q^2)_n(ab,abq^2;q^4)_n}\left(-\frac{q^2}{b}\right)^nq^{\binom{n+1}{2}} \nonumber \\
&=\frac{(aq^2,q^3/b;q^2)_\infty (ab^2,q^9/b^3;q^6)_\infty}{(ab,q^5/b^2;q^2)_\infty (q^3,aq^6/b;q^6)_\infty}. \label{eq-Rahman}
\end{align}
This formula also appears as \cite[Exercise 3.29(i)]{GR-book} except that a typo has been corrected here.

Setting $a=0$ in \eqref{eq-Rahman}, we obtain
\begin{align}\label{id-key}
\sum_{n=0}^\infty \frac{(q^2/b;q)_n(b;q^2)_n}{(q^5/b^2;q^2)_n(q^3;q^3)_n}(-1)^n \left(-\frac{q^2}{b}\right)^nq^{n(n+1)/2}=\frac{(q^3/b;q^2)_\infty (q^9/b^3;q^6)_\infty}{(q^5/b^2;q^2)_\infty (q^3;q^6)_\infty}.
\end{align}
Setting $b=q$ we complete the proof of Theorem \ref{thm-Wangconj}. In fact, replacing $b$ by $q^2/b$ we recover Chern's formula \eqref{Chern-gen}.
\end{proof}

\begin{proof}[Proof of Theorem \ref{thm-conj}]
We consider the series
\begin{align}
\overline{R}(q)&:=\oint \sum_{i=0}^\infty \frac{q^{\frac{1}{2}(i^2+3i)}(-1)^iz^i}{(q;q)_i}\sum_{j=0}^\infty \frac{z^{3j}q^{4j}}{(q^3;q^3)_j}\sum_{k=-\infty}^\infty (-1)^kz^{-k}q^{\frac{1}{2}(k^2-k)} \frac{dz}{2\pi iz} \nonumber \\
&=\sum_{i,j\geq 0} \frac{(-1)^jq^{i^2+3ij+\frac{9}{2}j^2+i+\frac{5}{2}j}}{(q;q)_i(q^3;q^3)_j}.
\end{align}
Now using \eqref{Euler-1}, \eqref{Euler-2} and \eqref{Jacobi} we deduce that
\begin{align}
\overline{R}(q)=\oint \frac{(q^2z,qz,1/z,q;q)_\infty}{(q^4z^3;q^3)_\infty} \frac{dz}{2\pi iz} =\oint \frac{(q^2z,qz,1/z,q;q)_\infty}{(q^{\frac{4}{3}}z,\omega q^{\frac{4}{3}}z,\omega^2 q^{\frac{4}{3}}z;q)_\infty} \frac{dz}{2\pi iz} \label{barR-int}
\end{align}
Now setting
$$(A,B,C,D)=(2,1,3,0), ~ (a_1,a_2)=(q^2,q), ~ b_1=1, ~ (c_1,c_2,c_3)=(q^{\frac{4}{3}},\omega q^{\frac{4}{3}}, \omega^2q^{\frac{4}{3}})$$
in \eqref{eq-integral}, we deduce that
\begin{align}
\overline{R}(q)=& \frac{(q^{\frac{4}{3}},q^{\frac{2}{3}},q^{-\frac{1}{3}};q)_\infty}{(\omega,\omega^2;q)_\infty} \overline{R}_1(q)+ \frac{(\omega q^{\frac{4}{3}}, \omega^2q^{\frac{2}{3}}, \omega^2q^{-\frac{1}{3}};q)_\infty}{(\omega^2,\omega;q)_\infty}\overline{R}_2(q) \nonumber \\
&+\frac{(\omega^2 q^{\frac{4}{3}}, \omega q^{\frac{2}{3}}, \omega q^{-\frac{1}{3}};q)_\infty}{(\omega,\omega^2;q)_\infty}\overline{R}_3(q) \nonumber \\
=& -\frac{1}{3}q^{-\frac{1}{3}}\frac{(q;q)_\infty}{(q^3;q^3)_\infty} \Big( (q^{\frac{1}{3}},q^{\frac{2}{3}},q^{\frac{2}{3}};q)_\infty \overline{R}_1(q)+\omega^2 (\omega q^{\frac{1}{3}}, \omega^2 q^{\frac{2}{3}}, \omega^2 q^{\frac{2}{3}};q)_\infty \overline{R}_2(q) \nonumber \\
& \qquad \qquad \qquad \quad +\omega (\omega^2 q^{\frac{1}{3}},\omega q^{\frac{2}{3}},\omega q^{\frac{2}{3}};q)_\infty \overline{R}_3(q)\Big), \label{barR-eval-key}
\end{align}
where
\begin{align}
\overline{R}_1(q)&=\sum_{n=0}^\infty \frac{(-1)^n(q^{\frac{1}{3}};q)_n q^{\frac{1}{2}n^2+\frac{5}{6}n}}{(q,\omega q,\omega^2q;q)_n}, \label{barR-1defn} \\
\overline{R}_2(q)&=\sum_{n=0}^\infty \frac{(-1)^n(\omega q^{\frac{1}{3}};q)_n\omega^n q^{\frac{1}{2}n^2+\frac{5}{6}n}}{(q,\omega q, \omega^2q;q)_n}, \label{barR-2defn}\\
\overline{R}_3(q)&=\sum_{n=0}^\infty \frac{(-1)^n(\omega^2 q^{\frac{1}{3}};q)_n \omega^{2n} q^{\frac{1}{2}n^2+\frac{5}{6}n}}{(q,\omega q,\omega^2 q;q)_n}. \label{barR-3defn}
\end{align}

By Theorem \ref{thm-Wangconj} we have
\begin{align}
\overline{R}_1(q^3)=\frac{(q^4;q^6)_\infty (q^{12};q^{18})_\infty}{(q^5;q^6)_\infty (q^9;q^{18})_\infty}. \label{barR-1-exp}
\end{align}
If we replace $q$ by $\omega q$ or $\omega^2 q$ in the definition of $\overline{R}_1(q^3)$, then we get $\overline{R}_2(q^3)$ and $\overline{R}_3(q^3)$, respectively. Thus
\begin{align}
\overline{R}_2(q^3)&=\frac{(\omega q^4;q^6)_\infty (q^{12};q^{18})_\infty}{(\omega^2 q^5;q^6)_\infty (q^9;q^{18})_\infty}, \label{barR-2-exp} \\
\overline{R}_3(q^3)&=\frac{(\omega^2q^4;q^6)_\infty (q^{12};q^{18})_\infty}{(\omega q^5;q^6)_\infty (q^9;q^{18})_\infty}. \label{barR-3-exp}
\end{align}
Now replacing $q$ by $q^3$ in \eqref{barR-eval-key} and substituting \eqref{barR-1-exp}--\eqref{barR-3-exp} into it, we obtain
\begin{align}
\overline{R}(q^3)=-\frac{1}{3}\frac{(q^3;q^3)_\infty (q^{12};q^{18})_\infty }{(q^9;q^9)_\infty (q^9;q^{18})_\infty} \left(W(q)+W(\omega q)+W(\omega^2q) \right), \label{barR-key}
\end{align}
where
\begin{align}
W(q)=q^{-1}(q,q^2,q^2;q^3)_\infty \frac{(q^4;q^6)_\infty}{(q^5;q^6)_\infty}=\frac{f_1f_2}{qf_3f_6}.
\end{align}
Recall the cubic continued fraction studied by Ramanujan (see page 366 of \cite{LostNotebook}):
\begin{align}\label{cubic-defn}
\nu(q):=\cfrac{q^{1/3}}{1+\cfrac{q+q^2}{1+\cfrac{q^2+q^4}{1+\cdots}}}.
\end{align}
Chan \cite{Chan-cubic} proved that
\begin{align}
\frac{1}{\nu(q^3)}-1-2\nu(q^3)=\frac{(q;q)_\infty (q^2;q^2)_\infty}{q(q^9;q^9)_\infty (q^{18};q^{18})_\infty}. \label{cubic-id}
\end{align}
If we write the series expansion of $W(q)$ as $W(q)=\sum_{n=-1}^\infty a(n)q^n$, then by \eqref{barR-key} and \eqref{cubic-id} we deduce that
\begin{align}
\overline{R}(q^3)=-\frac{(q^3;q^3)_\infty (q^{12};q^{18})_\infty}{(q^9;q^9)_\infty (q^9;q^{18})_\infty} \sum_{n=0}^\infty a(3n) q^{3n}=\frac{(q^{12},q^{18};q^{18})_\infty}{(q^9;q^{18})_\infty (q^6;q^6)_\infty}.
\end{align}
Replacing $q$ by $q^{\frac{1}{3}}$, we obtain \eqref{AU-conj}.
\end{proof}

\begin{proof}[Proof of Theorem \ref{thm14}]
Let
\begin{align}
F(u,v;q):=\sum_{i,j\geq 0} \frac{u^{i}v^{j}q^{\binom{i}{2}+2\binom{i+2j}{2}}}{(q;q)_{i}(q^4;q^4)_{j}}. \label{eq-14-F}
\end{align}
By \eqref{Euler-1}, \eqref{Euler-2} and \eqref{Jacobi}, this double sum satisfies
\begin{align}
F(u,v;q)&=\oint \sum_{i=0}^\infty \frac{(-1)^iu^{i}z^{i}q^{\binom{i}{2}}}{(q;q)_{i}} \sum_{j=0}^\infty \frac{v^{j}z^{2j}}{(q^4;q^4)_{j}} \sum_{k=-\infty}^\infty  (-1)^k q^{2\binom{k}{2}}z^{-k} \frac{\diff z}{2\pi iz}  \nonumber \\
&= \oint \frac{(uz;q)_\infty (q^2z,1/z,q^2;q^2)_\infty}{(vz^2;q^4)_\infty} \frac{\diff z}{2\pi iz} \nonumber \\
&= \oint \frac{(uz,uqz,q^2z,1/z,q^2;q^2)_\infty}{(v^{\frac{1}{2}}z,-v^{\frac{1}{2}}z;q^2)_\infty } \frac{\diff z}{2\pi iz}. \label{eq-14-integral}
\end{align}

Taking $(u,v)=(x,x^2)$, from \eqref{eq-14-F} and \eqref{eq-14-integral} we have
\begin{align}
F(x,x^2;q)&=\sum_{i,j\geq 0}\frac{q^{\frac{3}{2}(i^2-i)+4ij+4j^2-2j}x^{i+2j}}{(q;q)_{i}(q^4;q^4)_{j}} \nonumber \\
&= \oint \frac{(xqz,q^2z,1/z,q^2;q^2)_\infty}{(-xz;q^2)_\infty } \frac{\diff z}{2\pi iz}.
\end{align}
By the $q$-binomial theorem \eqref{q-binomial}, we deduce that
\begin{align}\label{Fx}
F(x,x^2;q)&=\oint \sum_{n=0}^\infty \frac{(-q;q^2)_n(-xz)^n}{(q^2;q^2)_n}\sum_{k=-\infty}^\infty (-1)^kq^{k(k-1)}z^{-k} \frac{\diff z}{2\pi iz} \nonumber \\
&=\sum_{n=0}^\infty \frac{q^{n^2-n}(-q;q^2)_nx^n}{(q^2;q^2)_n}.
\end{align}

Setting $x=q$ in \eqref{Fx}, we obtain by \eqref{Slater36} that
\begin{align}
F(q,q^2;q)=\sum_{n=0}^\infty \frac{q^{n^2}(-q;q^2)_n}{(q^2;q^2)_n}=\frac{1}{(q,q^4,q^7;q^8)_\infty}.
\end{align}
This proves \eqref{KR21}.

Setting $x=q^3$ in \eqref{Fx}, we obtain by \eqref{Slater34} that
\begin{align}
F(q^3,q^6;q)=\sum_{n=0}^\infty \frac{q^{n^2+2n}(-q;q^2)_n}{(q^2;q^2)_n}=\frac{1}{(q^3,q^4,q^5;q^8)_\infty}.
\end{align}
This proves \eqref{KR22}.

By \eqref{Fx} and \eqref{Lebesgue} with $(a,q)\rightarrow (-q^3,q^2)$, we see that the left side of \eqref{KR23} is
\begin{align}
&F(q^2,q^4;q)+qF(q^4,q^8;q)=\sum_{n=0}^\infty \frac{q^{n^2+n}(-q;q^2)_n}{(q^2;q^2)_n}(1+q^{2n+1}) \nonumber  \\
&=(1+q)\sum_{n=0}^\infty \frac{q^{n(n+1)}(-q^3;q^2)_{n}}{(q^2;q^2)_n}  \nonumber \\
&=(-q;q^4)_\infty (-q^2;q^2)_\infty =\frac{1}{(q,q^5,q^6;q^8)_\infty}.  \nonumber
\end{align}
This proves \eqref{KR23}.

Setting $x=q^2$ in \eqref{Fx}, we obtain by \eqref{Gollnitz24} that
\begin{align}
F(q^2,q^4;q)=\sum_{n=0}^\infty \frac{q^{n^2+n}(-q;q^2)_n}{(q^2;q^2)_n}=\frac{1}{(q^2,q^3,q^7;q^8)_\infty}.
\end{align}
This proves \eqref{KR24}.
\end{proof}
\begin{rem}
In the proof of \eqref{KR23} we have shown that
\begin{align}
\sum_{n=0}^\infty \frac{q^{n(n+1)}(-q;q^2)_{n+1}}{(q^2;q^2)_n}=\frac{1}{(q,q^5,q^6;q^8)_\infty}. \label{Goll-new}
\end{align}
This can be compared with \eqref{Gollnitz22}. They have different sum sides but the same product sides.  It is easy to see that \eqref{Gollnitz22} is also a special case of \eqref{Lebesgue} with $q$ replaced by $q^2$ and $a=-q^{-1}$.
\end{rem}

\section{Concluding remarks}\label{sec-conclude}
So far we have seen that the integral method is quite useful in proving certain Rogers-Ramanujan type identities. It should be mentioned that this method has its limitations. When the sum side cannot be expressed in terms of a nice integral which is calculable, this method may not work.

Besides the triple sum identities including \eqref{KR-triple-1} and \eqref{KR-tripke-2},  Kanade and Russell \cite{KR-2019} also discovered some conjectural quadruple sum identities. To state their conjectures, we shall define
\begin{align}
F(u,v,w,t;q):=\sum_{i,j,k,\ell \geq 0}\frac{(-1)^{\ell}u^{i}v^{j}w^{k}t^{\ell}q^{\binom{i+2j+3k+4\ell}{2}+2\ell^2-2\ell}}{(q;q)_{i}(q^2;q^2)_{j}(q^3;q^3)_{k}(q^4;q^4)_{\ell}}.
\end{align}
The equivalent version of the conjectures of Kanade and Russell \cite[Eqs.\ (62)--(65)]{KR-2019} are as follows:
\begin{align}
F(q,q^3,q^6,q^8;q)&=\frac{1}{(q,q^3,q^4,q^6,q^8,q^9,q^{11};q^{12})_\infty}, \label{KR-conj-1} \\
F(q^3,q^5,q^6,q^{12};q)&=\frac{1}{(q^3,q^4,q^5,q^6,q^7,q^8,q^9;q^{12})_\infty}, \label{KR-conj-2} \\
F(q^2,q^3,q^5,q^8;q)&=\frac{1}{(q^2,q^3,q^4,q^5,q^8,q^9,q^{11};q^{12})_\infty} \label{KR-conj-3}
\end{align}
and
\begin{align}
&F(q^2,q^3,q^4,q^8;q)+qF(q^3,q^5,q^7,q^{10};q)+q^3F(q^4,q^7,q^{14},q^{10};q)\nonumber \\
& \quad -q^7F(q^5,q^9,q^{18},q^{13};q)  =\frac{1}{(q,q^3,q^4,q^7,q^8,q^9,q^{10};q^{12})_\infty}. \label{KR-conj-4}
\end{align}
With the help of Maple, we searched for similar identities and find the following identities:
\begin{align}
F(q,q^3,q^4,q^6;q)&=\frac{1}{(q,q^3,q^4,q^6,q^7,q^{10},q^{11};q^{12})_\infty}, \label{conj-1} \\
F(q^2,q^5,q^5,q^{10};q)&=\frac{1}{(q^2,q^3,q^5,q^6,q^7,q^8,q^{11};q^{12})_\infty}. \label{conj-2}
\end{align}
It turns out that they were already implicitly included as conjectures in the work of Kanade and Russell \cite{KR-2019}. Indeed, \eqref{conj-1} (resp.\ \eqref{conj-2}) can be viewed as an analytic form of the conjectural partition identity $I_5$ (resp.\ $I_6$) in \cite{KR-2015} (see also \cite[Secs.\ 5.1.1, 5.1.2]{KR-2019}). In fact,  setting $x=1$ in \cite[Eq.\ (54)]{KR-2019}, we get (conjecturally) \eqref{conj-2}. One can treat the second line of \cite[Eq.\ (45)]{KR-2015} in a way similar to \cite[Eqs.\ (52),(53)]{KR-2019} and then get an expression similar to \cite[Eq.\ (54)]{KR-2019}, and then we get \eqref{conj-1} (conjecturally) after setting $x=1$ in the resulting expression.

Furthermore, Kanade and Russell \cite{KR-2019} also gave different analytical forms for the conjectural partitions identities $I_5$ and $I_6$. After setting $x=1$ in \cite[Eqs.\ (47),(51)]{KR-2019} (except for a typo there), we see that the conjectural partition identities $I_5$ and $I_6$ are equivalent to the following triple sum identities
\begin{align}
\sum_{i,j,k\geq 0}\frac{q^{\binom{i+2j+3k}{2}+i+j^2+2j+4k}}{(q;q)_i(q^2;q^2)_j(q^3;q^3)_k}&=\frac{1}{(q,q^3,q^4,q^6,q^7,q^{10},q^{11};q^{12})_\infty}, \label{conj-1-equivalent} \\
\sum_{i,j,k\geq 0}\frac{q^{\binom{i+2j+3k}{2}+2i+j^2+j+4j+5k}}{(q;q)_i(q^2;q^2)_j(q^3;q^3)_k}&=\frac{1}{(q^2,q^3,q^5,q^6,q^7,q^8,q^{11};q^{12})_\infty}. \label{conj-2-equivalent}
\end{align}
They have already been proved by Bringmann, Jennings-Shaffer and Mahlburg \cite{BSM}. Hence the partition identities $I_5$ and $I_6$ have been confirmed and so the conjectural identities \eqref{conj-1} and \eqref{conj-2} are also confirmed.

So far the conjectural identities \eqref{KR-conj-1}--\eqref{KR-conj-4} are still open. We believe that the identities \eqref{KR-conj-1}--\eqref{conj-2} can all be proved via the integral method. In fact, we can use \eqref{Euler-1}, \eqref{Euler-2} and \eqref{Jacobi} to express the sum side as
\begin{align}
F(u,v,w,t;q)
&=\oint \sum_{i=0}^\infty \frac{(-uz)^i}{(q;q)_i} \sum_{j=0}^\infty \frac{(vz^2)^j}{(q^2;q^2)_j} \sum_{k=0}^\infty \frac{(-wz^3)^k}{(q^3;q^3)_k} \sum_{\ell=0}^\infty \frac{q^{2\ell^2-2\ell}(-tz^4)^\ell}{(q^4;q^4)_\ell} \nonumber \\
& \qquad \qquad  \times \sum_{s=-\infty}^\infty (-1)^s q^{\binom{s}{2}}z^{-s} \frac{dz}{2\pi iz} \nonumber \\
&=\oint \frac{(tz^4;q^4)_\infty (qz,1/z,q;q)_\infty}{(-uz;q)_\infty (vz^2;q^2)_\infty (-wz^3;q^3)_\infty} \frac{dz}{2\pi iz} \label{QF-1} \\
&=\oint \frac{(t^{\frac{1}{4}}z,-t^{\frac{1}{4}}z,it^{\frac{1}{4}}z,-it^{\frac{1}{4}}z,qz,1/z,q;q)_\infty}{(-uz,v^{\frac{1}{2}}z,-v^{\frac{1}{2}}z,-w^{\frac{1}{3}}z,-\omega w^{\frac{1}{3}}z,-\omega^2 w^{\frac{1}{3}}z;q)_\infty} \frac{dz}{2\pi iz}. \label{QF-2}
\end{align}
It is still possible to evaluate this integral using Lemma \ref{lem-integral} and hence give a way to prove the above conjectures of Kanade and Russell. We have made some progress towards these conjectures and plan to discuss this in a separate paper.

It is clear that the integral method can also be applied to discover new identities. In another paper \cite{Cao-Wang}, using this method, we will establish some new multi-sum Rogers-Ramanujan type identities with various indexes.

\subsection*{Acknowledgements}
We thank Matthew Russell for pointing out the relations between \eqref{conj-1}--\eqref{conj-2} and the conjectural identities $I_5$ and $I_6$ in \cite{KR-2015} and their discussions of these identities in Sections 5.1.1 and 5.1.2 in \cite{KR-2019}. This work was supported by the National Natural Science Foundation of China (12171375).

\end{document}